%% file: filterdcomp.tex
\documentclass{amsart}

\pdfoutput=1

\usepackage{amsmath}
\usepackage{amssymb}
\usepackage{amsthm}
\usepackage{mathabx}
\usepackage{xypic}
\usepackage{graphicx}
\usepackage{color}

\usepackage{layouts}

\newtheorem{lemma}{Lemma}[section]
\newtheorem{theorem}[lemma]{Theorem}
\newtheorem{cor}[lemma]{Corollary}
\newtheorem{prop}[lemma]{Proposition}
\newtheorem{defn}[lemma]{Definition}

\newtheorem{conj}[lemma]{Conjecture}

\newcommand{\Z}{\mathbb{Z}}
\newcommand{\Q}{\mathbb{Q}}
\newcommand{\R}{\mathbb{R}}

\newcommand{\CKh}{\operatorname{CKh}^{\pm}}

\title{A family of concordance homomorphisms from Khovanov homology}
\author{William Ballinger}

\begin{document}

\maketitle

\begin{abstract}
By considering a version of Khovanov homology incorporating both the Lee and $E(-1)$ differentials, we construct a $1$-parameter family of concordance homomorphisms similar to the Upsilon invariant from knot Floer homology. This invariant gives lower bounds on the slice genus and can be used to prove that certain infinite families of pretzel knots are linearly independent in the smooth concordance group. 
\end{abstract}

\section{Introduction}

The Upsilon invariant, introduced by Ozsv\'ath, Stipsicz, and Szab\'o in \cite{ozsvath2017concordance}, is defined in terms of the knot Floer homology of a knot and contains a great deal of information about the smooth concordance group. The definition, especially as formulated by Livingston \cite{livingston2017notes}, can be abstracted to avoid reference to any details of the construction of knot Floer homology, only relying on the fact that it associates a chain complex with two independent filtrations to any knot. This suggests an approach to creating new knot invariants by constructing a similar filtered complex in a different setting, which was used for example in \cite{lewark2019upsilon} to define knot invariants using a second filtration on a version of the Khovanov-Rozansky $\operatorname{sl}(n)$ homology. This paper will give a new way to define Upsilon-like invariants of knots, this time using the Lee and $E(-1)$ differentials on Khovanov ($\operatorname{sl}(2)$) homology. Most of the results here should extend to the setting of $\operatorname{sl}(n)$ homology, since analogues of the Lee and $E(-1)$ differentials are also present there. Making this extension would require adapting the invariance proofs from \cite{ballinger2020concordance} to the $\operatorname{sl}(n)$ setting, so we will focus on the $\operatorname{sl}(2)$ case for simplicity and ease of computation.

\begin{theorem}
There is a knot invariant $\Phi_K(\alpha)$ for $K$ a knot in $S^3$ and $\alpha \in [0,2]$ with the following properties:
\begin{itemize}
\item For fixed $\alpha$, $\Phi_K(\alpha)$ is a real-valued concordance homomorphism. 
\item For fixed $K$, $\Phi_K(\alpha)$ is a continuous piecewise linear function of $\alpha$. 
\item $|\Phi_K(\alpha)| \le 2g_4(K) \Lambda(\alpha)$, where $\Lambda(\alpha) = 1-|\alpha-1|$.
\end{itemize}
\end{theorem}
The above properties are all shared with $\Upsilon$, but notably there is no analogue of the symmetry $\Upsilon_K(t) = \Upsilon_K(2-t)$. Another property shared with $\Upsilon$ is that the lines forming the graph of $\Phi_K$ are constrained by the Khovanov homology of $K$:
\begin{theorem}\label{homologyconstraint}
Suppose $L$ is a line meeting the graph of $\Phi_K$ in a segment, and let $\ell_0$ and $\ell_1$ be the second coordinates of the intersections of $L$ with the vertical lines $\alpha = 0$ and $\alpha = 1$, respectively. Then the reduced Khovanov homology of $K$ is nonzero in homological grading $\ell_0/2$ and delta grading $\ell_1$. In particular, $\ell_0$ and $\ell_1$ are even integers.
\end{theorem} 

For some simple knots $K$, the function $\Phi_K$ can be computed by directly examining the Khovanov homology. In particular, this can be done for the three stranded pretzel knots based on computations of their reduced Khovanov homology due to Manion \cite{manion2018khovanov}, leading to conclusions about their place in the smooth concordance group.
\begin{theorem}\label{pretzelindep}
For any two sequences $a_k,b_k$ of odd integers with $2k < a_k < b_k$, the sequence of knots $\{P(-2k,a_k,b_k)\}_{k=1}^\infty$ are a basis for a $\Z^\infty$ summand of the smooth concordance group. 
\end{theorem}
None of the knots appearing in this theorem are algebraically slice, and it seems likely that this theorem, and even stronger statements, could be proven using other concordance invariants such as the Tristam-Levine signature. However, this would require a detailed analysis of the roots of the Alexander polynomials of these knots while, once the basic properties of the invariant are set up and the Khovanov homology of these knots is known, finding their $\Phi$ invariants is a fairly simple calculation.

The invariant $\Phi_K$ can also be calculated for small torus knots using their known Khovanov homologies (taken from \cite{bar2005knot}). Despite the complexity of these homology groups, a relatively simple pattern emerges which leads to Conjecture~\ref{torusvalues} on the values of $\Phi_{T_{p,q}}$ for any $p,q$.

\section{A bifiltered complex from a knot}

In \cite{ballinger2020concordance}, for every knot $K$ a $\Z$-graded, $\Z$-filtered chain complex over $\Z[x]$ $C(K)$ was defined. In this complex, the differential lowers the grading by $3$ and multiplication by $x$ lowers it by $2$. The following properties of $C(K)$, proven in \cite{ballinger2020concordance}, will be relavant here:
\begin{prop}\label{elementaryE-1}
Forgetting the filtration, the homology of $C(K)$ is isomorphic to $\Z[x]$, and the $E_2$ page of the spectral sequence coming from $C(K)$ is the $\mathcal{F}_3$ Khovanov homology of $K$ as defined in \cite{khovanov2004link}. If $K$ and $K'$ are related by a cobordism of genus $g$, there is a filtered chain map $C(K) \to C(K')$ whose induced map on homology is multiplication by $(2x)^g$, and if $K_-$ and $K_+$ are related by a crossing change in which $K_-$ has a negative crossing and $K_+$ a positive crossing, there is a filtered chain map $C(K_-) \to C(K_+)$ inducing an isomorphism on homology.
\end{prop}
The above identification between the $E_2$ page and the Khovanov homology relates the filtration on $C(K)$ to the homological grading, and the grading on $C(K)$ to the linear combination $q-3h$ of the quantum and homological grading. We will call $q-3h$ the internal grading on Khovanov homology, just as in \cite{ballinger2020concordance}.

This paper will be concerned with the quotient of $C(K)$ setting the variable $x$ to $1$. To make the resulting invariants better behaved, we will also change base to $\Q$, although any field of characteristic not $2$ would work just as well and produce a different invariant. The result is a bifiltered chain complex incorporating both the Lee (or $d_1$) and $d_{-1}$ differentials into Khovanov homology. 
\begin{defn}
For a knot $K$, let $\CKh(K)$ be the chain complex $(C(K) \otimes \Q)/(x-1)$, and for integers $i,j$ let $\CKh(K)_{i,j}(K)$ be the subgroup of $\CKh(K)$ generated by the images of elements of $C(K)$ in grading at least $i$ and filtration level at least $j$.
\end{defn}
Since $x$ has degree $2$, in addition to the two filtrations $\CKh(K)$ inherits a $\Z/2\Z$ grading from the grading on $C(K)$. This will not be needed for establishing the formal properties of the invariant, but can be useful in computation. This complex has the following elementary properties, immediately from Proposition~\ref{elementaryE-1}:
\begin{prop}\label{elementaryprop}
The differential on $\CKh(K)$ carries $\CKh(K)_{i,j}(K)$ into $\CKh(K)_{i-3,j}(K)$, and the homology of $\CKh(K)$ with respect to this differential is isomorphic to $\Q$. If $K$ and $K'$ are related by a cobordism of genus $g$, then there is a chain map $\CKh(K) \to \CKh(K')$ that sends $\CKh_{i,j}(K)$ into $\CKh_{i-2g,j}(K')$ and induces an isomorphism on homology. If $K_-$ and $K_+$ are related by a crossing change in which $K_-$ has a negative crossing and $K_+$ a positive crossing, there is a filtered chain map $\CKh(K_-) \to \CKh(K_+)$ inducing an isomorphism on homology. 
\end{prop}

Since the differential on $\CKh(K)$ lowers the first of the filtrations, this isn't quite a filtered chain complex in the usual sense. However, most of what one might like to do to a filtered chain complex can still be done here. In particular, collapsing the two filtrations to one along a one-parameter family of directions produces a family of concordance homomorphisms $\Phi_K(\alpha)$ for $\alpha \in [0,2]$ very similar to the $\Upsilon$ invariant from Heegaard Floer homology. 

\begin{theorem}
For each $\alpha \in [0,2]$, $\Phi_K(\alpha)$ is a real valued concordance homomorphism, and $\Phi_K$ is a piecewise linear function of $K$. At the endpoints, $\Phi_K(0) = \Phi_K(2) = 0$, and the slopes of $\Phi_K$ at $0$ and $2$ are $s(K)$ and $-t(K)$, respectively.
\end{theorem}

Just like $\Upsilon$, the invariant $\Phi$ gives lower bounds on the slice genus of $K$. 
\begin{theorem}\label{slicegenus}
\begin{equation*}
|\Phi(K,\alpha)| \le 2g_4(K) \min(\alpha, 2-\alpha)
\end{equation*}
\end{theorem}
Note that, although the slice genus bound from $\Upsilon$ has the same shape as this one, in the case of $\Upsilon$ the symmetry $\Upsilon_K(t) = \Upsilon_K(2-t)$ means that all of the information is contained in the bound over the interval $[0,1]$. In the case of $\Phi$, however, it is possible that the bound on the interval $[1,2]$ has additional information, although this does not occur in any known example. 

\section{Preliminaries on homological algebra}

\begin{defn}
A semibifiltered complex is a chain complex $(C,d)$ over a field together with an increasing $\Z \oplus \Z$-filtration $C_{i,j}$ of $C$, with no assumption on how the differential $d$ interacts with the complex. A semibifiltered complex is $(m,n)$-bounded if $d(C_{i,j}) \subset C_{i+m,j+n}$. 
\end{defn}
In general, we will consider filtered chain maps between semibifiltered complexes, but not require that homotopies between these maps be filtered. In particular, a homotopy equivalence of semibifiltered complexes $C,C'$ consists of filtered chain maps $C \to C'$ and $C' \to C$, such that both composite maps are chain homotopic to the identity via not-necessarily-filtered homotopies. With this definition, the homotopy equivalence class of $\CKh(K)$ as a semibifiltered complex is an invariant of the concordance class of $K$. 

From an $(m,n)$-bounded semibifiltered complex $(C,d)$, one can construct associated graded complexes with respect to both filtrations, producing a $\Z \oplus \Z$-graded chain complex whose differential has degree $(m,n)$. Similarly, one can take the associated graded complex with respect to only the second filtration, producing a $\Z$-graded, $\Z$ filtered chain complex whose differential has degree $n$ with respect to the grading and changes the filtration by at least $m$. Call the first of these constructions the associated graded complex, and the second the vertical associated graded complex. Note that these concepts depend on the choice of $(m,n)$ - If $C$ is $(m,n)$-bounded, then it is also $(m',n')$-bounded for any $m' \le m$ and $n' \le n$, but treating $C$ as an $(m',n')$-bounded complex and taking the associated graded generally results in a complex with vanishing differential. 

A straightforward homological perturbation argument proves the following:
\begin{lemma}\label{spectralstep}
Any $(m,n)$-bounded semibifiltered complex $(C,d)$  is homotopy equivalent to an $(m,n)$-bounded semibifiltered complex $(C',d)$ such that the associated graded module of $C'$ is isomorphic to the homology of the associated graded complex of $(C,d)$. Furthermore, the vertical associated graded complex of $(C',d)$ depends only on the vertical associated graded complex of $(C,d)$. 
\end{lemma}

The complex $\CKh(K)$ is $(-3,0)$-bounded, and the homology of its associated graded complex is precisely the underlying vector space of the rational Khovanov chain complex, as originally defined. This lemma can therefore produce a model of $\CKh(K)$ with underlying vector space equal to the Khovanov chain complex. This is finite dimensional and already much more convinient for computations, but an even further simplification is possible:

\begin{lemma}\label{khsimplify}
The complex produced from $\CKh(K)$ via Lemma~\ref{spectralstep} is in fact $(-3,1)$-bounded, and its associated graded complex as a $(-3,1)$-bounded semibifiltered complex is precisely the Khovanov complex of $K$. 
\end{lemma}
\begin{proof}
To say that a $(-3,0)$-bounded complex is $(-3,1)$-bounded is exactly to say that the differential of it's vertical associated graded complex vanishes. By the second part of Lemma~\ref{spectralstep}, this can be checked looking only at the vertical associated graded complex of $\CKh(K)$, which (directly from the definition in \cite{ballinger2020concordance}) is precisely a direct sums of shifts of copies of $\CKh(U)$ where $U$ is an unlink with a varying number of components. The homology of such a complex is the reduced Lee homology of this unlink, and the homology of it's associated graded complex is the reduced Khovanov homology of this unlink, so since these have the same rank all higher vertical differentials vanish. Finally, the differential of the associated graded complex of this $(-3,-1)$-bounded complex is precisely the differential on the $E_1$-page of the spectral sequence considered in \cite{ballinger2020concordance}, where it is shown that this is the differential of the Khovanov complex.
\end{proof}

Therefore, applying Lemma~\ref{spectralstep} once more proves the following, which enables most of the computations to follow:
\begin{lemma}\label{fullsimplify}
$\CKh(K)$ is chain homotopy equivalent to a $(-3,1)$-bounded semibifiltered complex $(C,d)$ such that the associated graded module of the filtration on $C$ is isomorphic to the Khovanov homology $\operatorname{Kh}(K)$, with its internal and homological gradings. 
\end{lemma}

\begin{figure}[h]
  \centering
  \def\svgwidth{\columnwidth}
  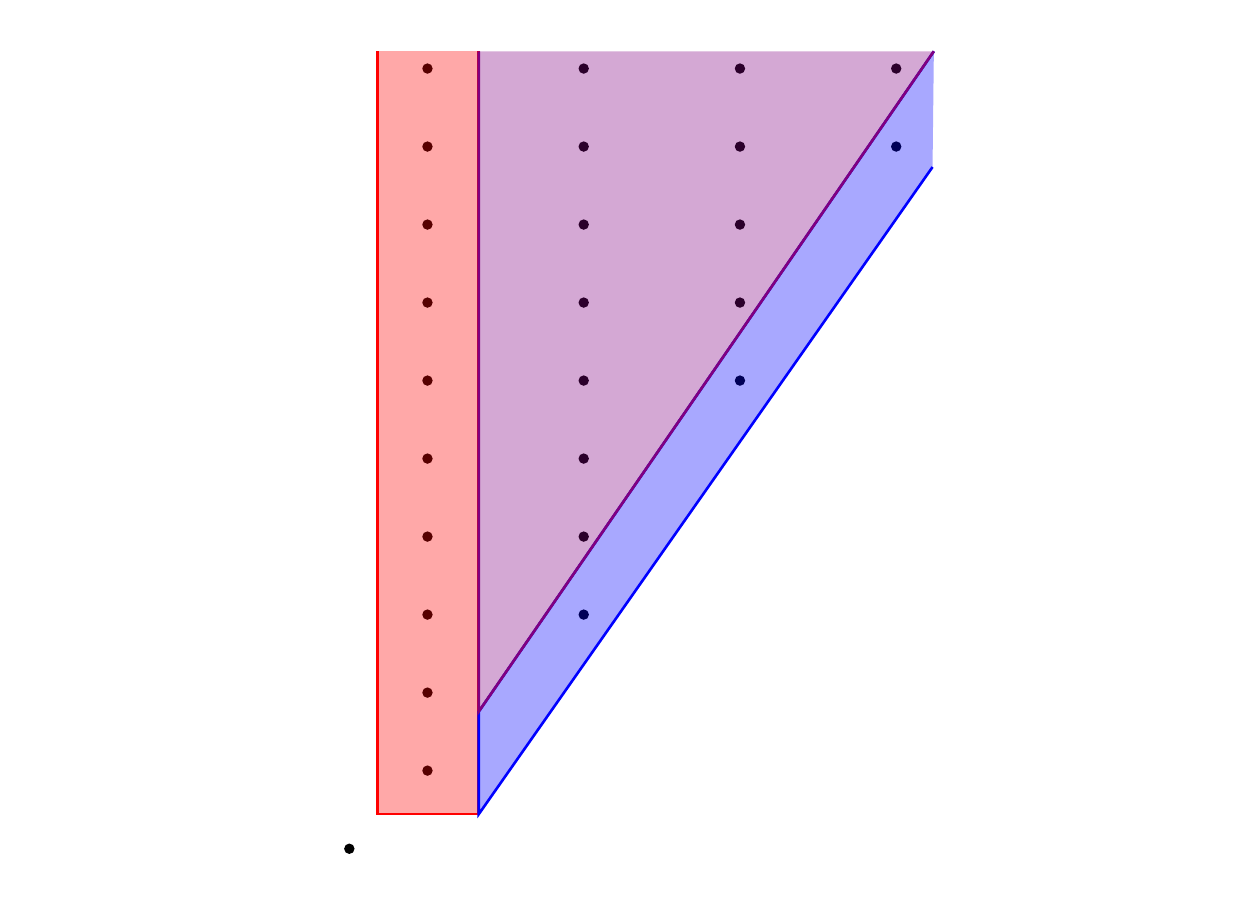
  \caption{The allowed differentials in the model for $\CKh(K)$ provided by Lemma~\ref{fullsimplify}. The complex $\CKh(K)$ has a model whose underlying vector space is the reduced Khovanov homology of $K$ with rational coefficients, and the differential of an element in the grading labelled $1$ in the diagram lies in the gradings marked by dots in the shaded regions. This diagram takes into account the fact that the differential must increase the homological grading $h$ by at least $1$, decrease the internal grading $q - 3h$ by at most $3$, and change the parity of the internal grading, as well as the fact that no differential of degree $h$ will be present due to the last application of Lemma~\ref{spectralstep} passing from the Khovanov complex to the Khovanov homology. The terms of the differential in the red region will be exactly the Lee differential of $K$, and similarly the terms in the blue region will be the $E(-1)$ differential of $K$.}\label{alloweddifferentials}
\end{figure}

\section{Invariants of filtered complexes}

Let $C$ be a decreasing $\R$-filtered vector space with filtered pieces $\{C_t\}_{t \in \R}$ and $d$ a differential on $C$, with no assumption on how $d$ interacts with the filtration. If the homology of $(C,d)$ is one dimensional, define $m(C)$ to be the supremum of all $t$ for which there is an element $\gamma$ in $C_t$ with $d\gamma = 0$ and $[\gamma]$ nonzero in homology. 

\begin{lemma}\label{additivity}
For $C$, $C'$ as above, $m(C \otimes C') = m(C) + m(C')$. 
\end{lemma}
\begin{proof}
For $t \in \R$, define $C_{> t}$ to be the union of $C_s$ for all $s > t$. Since the homology of $C$ is one-dimensional, any two chains representing nonzero homology classes have homologous scalar multiples. Therefore, the quantity $m(C)$ is characterized by the fact that, for all $t < m(C)$, there is $\gamma \in C_t$ with $d\gamma = 0$ and $\gamma \not\in d(C) + C_{>m(C)}$. Fix $t < m(C)$ and $t' < m(C')$, and elements $\gamma$ and $\gamma'$ with this property. Then $\gamma \otimes \gamma' \in C_t \otimes C_{t'} \subset (C \otimes C')_{t + t'}$, and is not contained in the subspace
\begin{equation*}
d(C) \otimes C' + C_{>m(C)} \otimes C' + C \otimes d(C') + C \otimes C'_{>m(C')}
\end{equation*}
Now, $d(C \otimes C')$ is a subspace of $d(C) + d(C')$, and $C_{>m(C)} \otimes C' + C \otimes C'_{>m(C')}$ is a subspace of $(C \otimes C')_{> m(C) + m(C')}$, so $\gamma \otimes \gamma'$ is not contained in $d(C \otimes C') + (C \otimes C')_{> m(C) + m(C')}$, so taking $t$ and $t'$ arbitrarily close to $m(C)$ and $m(C')$, respectively, finishes the proof. 
\end{proof}

If now $(C,d)$ is a semibifiltered complex with one-dimensional homology, then for any $\alpha \in [0,2]$ we can collapse the filtration to one $\R$-filtration and apply the above construction: let
\begin{equation*}
C^\alpha_t = \sum_{ \alpha m + (2-\alpha) n \ge t} C_{m,n},
\end{equation*}
and let $M(C,\alpha) = m(C^\alpha)$. Then by the additivity of $m$, $M(C \otimes C',\alpha) = M(C,\alpha) + M(C',\alpha)$. 

The complex $\CKh(K)$ is semibifiltered with one-dimensional homology, so the above definition applies:
\begin{defn}
For a knot $K$ and $\alpha \in [0,2]$, let $\Phi_K(\alpha) = M(\CKh(K),\alpha)$. 
\end{defn}

Applying this construction to an appropriate version of knot Floer homology produces the Upsilon invariant. Via arguments similar to Theorem 8.1 in \cite{livingston2017notes}, analogues of the integrality properties of Upsilon can be proven in this more general context.

\begin{prop}
Suppose that $(C,d)$ is a semibifiltered complex with $C$ finite dimensional. Then the function $M(C,-)$ is continuous and piecewise-linear, and it's graph is contained in the union of the lines passing through $(0,2j)$ and $(2,2i)$ as $(i,j)$ ranges over the support of the associated graded module to $C$. 
\end{prop}
\begin{proof}
For fixed $\alpha$, the quantity $M(C,\alpha)$ must lie in the support of the associated graded module to the collapsed filtration $C^\alpha$ since otherwise any chain lying in filtration level $M(C,\alpha)$ would in fact lie in a deeper level. As $\alpha$ varies, the support of this associated module traces out the lines in the statement, so it remains only to prove that $M(C,-)$ is a continuous function of $\alpha$. For a fixed chain $\gamma \in C$, let $M(\gamma,\alpha)$ be the largest $t$ for which $\gamma$ is contained in $C^\alpha_t$. This will be the minumum of several linear functions corresponding to the support of $\gamma$, so will be continuous and piecewise linear. The quantity $M(C,\alpha)$ is the maximum of $M(\gamma,\alpha)$ over all $\gamma$ representing a nonzero homology class. Since each $M(\gamma,\alpha)$ is contained in the lines corresponding to the support of $C$, of which there are finitely many, there are only finitely many possible values for $M(\gamma,\alpha)$. Therefore, $M(C,\alpha)$ is the maximum of finitely many continuous functions, so is itself continuous. 
\end{proof}

Combining this with Lemma~\ref{fullsimplify} immediately proves Theorem~\ref{homologyconstraint}.

The quantity $M(C,\alpha)$ behaves in a predictable way under chain maps, which will be important to establish the basic properties of $\Phi_K$.

\begin{lemma}\label{boundingmap}
If $C$ and $C'$ are semibifiltered complexes and there is a filtered chain map $f: C \to C'$ inducing an isomorphism on homology, then $M(C,\alpha) \le M(C',\alpha)$. 
\end{lemma}
\begin{proof}
For any $t < M(C,\alpha)$, there is a chain $\gamma \in C^\alpha_t$ representing a nonzero homology class. Then $f(\gamma) \in (C')^\alpha_t$, so $t \le M(C',\alpha)$. By taking $t$ approaching $M(C,\alpha)$, then $M(C,\alpha) \le M(C',\alpha)$. 
\end{proof}

\subsection{Basic Properties of $\Phi_K(\alpha)$}

This section contains proofs of most of the theorems stated in the introduction, as well as a few additional facts needed for later computations. 

\begin{prop}
$\Phi$ is a concordance homomorphism.
\end{prop}
\begin{proof}
The complex $C(K)$ from \cite{ballinger2020concordance} satisfies $C(K_1 \# K_2) = C(K_1) \otimes_{\Z[x]} C(K_2)$, so $\CKh(K_1 \# K_2) \cong \CKh(K_1) \otimes_\Q \CKh(K_2)$.  Additivity of $\Phi$ now follows from Lemma~\ref{additivity}. To see that $\Phi$ is a concordance invariant, use Lemma~\ref{boundingmap} and the fact that the complexes $\CKh(K)$ for concordant knots admit filtered chain maps inducing isomorphisms on homology in either direction by Proposition~\ref{elementaryprop}.
\end{proof}

\begin{prop}
For any knot $K$, the values $\Phi_K(0)$ and $\Phi_K(2)$ are both $0$, and the slopes $\Phi_K'(0)$ and $\Phi_K'(2)$ are $s(K)$ and $-t(K)$, respectively. 
\end{prop}
\begin{proof}
In the model for $\CKh(K)$ provided by Lemma~\ref{fullsimplify}, the vertical associated graded complex of $\CKh(K)$ is precisely the reduced Khovanov homology of $K$ together with all differentials of the Lee spectral sequence, so the homology of this vertical associated graded complex is $\Q$ in homological degree $0$ and $0$ in all other homological degrees, and with respect to the filtration by the internal degree there is a generator for the homology contained in filtration level $s(K)$ but not in any deeper level of the filtration. Let $\gamma_0 \in \CKh(K)_{s(K),0}$ be any lift of a chain generating the homology of the verical associated graded complex. Since $\gamma$ is a chain in the vertical associated graded, $d\gamma_0 \in \CKh(K)_{*,2}$, and since the homology of the vertical associated graded in degree $2$ vanishes, there is a $\gamma_1 \in \CKh(K)_{*,1}$ with $d\gamma_0 + d\gamma_1 \in \CKh(K)_{*,3}$. Continuing in this way to push remaining terms depper into the filtration, and using the fact that this model of $\CKh(K)$ is finite dimensional, we produce a chain $\gamma \in \CKh(K)_{s(K),0} + \CKh(K)_{*,1}$ generating the homology of $\CKh(K)$ which, again using finite dimensionality, in fact lies in $\CKh(K)_{s(K),0} + \CKh(K)_{m,1}$ for some $m$. Therefore, $\Phi_K(\alpha) \ge s(K)\alpha$ for small enough $\alpha$, and the reverse inequality follows from applying the same argument to $-K$. Similarly, the horizontal associated graded complex of $\CKh(K)$ (taking associated graded with respect to just the first filtration) is exactly the reduced complex $\bar{C}(K)$ from \cite{ballinger2020concordance}, since taking associated graded changes the relation $x = 1$ to $x = 0$, so applying this same argument proves the statement about the right endpoint slope. 
\end{proof}
The version of the Lee spectral sequence from the above proof is not exactly the same as the one used by Rasmussen \cite{rasmussen2010khovanov} to define the $s$ invariant, but is instead a reduced version of the Lee spectral sequence. The version used in \cite{rasmussen2010khovanov} can be instead obtained by replacing the relation $x-1$ by $x^2 - 1$ and shifting all quantum gradings up by $1$ (and again taking the vertical associated graded to remove $E(-1)$ differentials). The unreduced complex has two dimensional homology with generators in filtration levels $s(K) - 1$ and $s(K) + 1$, and the necessary generator of the reduced homology in filtration level $s(K)$ just comes from further quotienting the unreduced complex by $x-1$ to get the reduced complex, and taking the image of the generator in filtration level $s(K) - 1$. This will land in filtration level $s(K)$ in the reduced complex, due to the grading shift.  

\begin{prop}\label{crossingchange}
If $K_-$ and $K_+$ differ by a crossing change in which $K_-$ has a negative crossing and $K_+$ has a positive crossing, then $\Phi_{K_-}(\alpha) \le \Phi_{K_+}(\alpha)$ for any $\alpha \in [0,2]$. 
\end{prop}
\begin{proof}
By Proposition~\ref{elementaryprop}, in this situation there is a filtered chain map $\CKh(K_-) \to \CKh(K_+)$ that induces an isomorphism on homology. The inequality then follows from Lemma~\ref{boundingmap}
\end{proof}

\section{Twisting ribbon surfaces and the slice genus bound}

Arguments similar to those of the previous section suffice to prove a slice genus bound $g_4(K) \ge \Phi_K(\alpha)/\alpha$, by considering the cobordism maps on $C(K)$ and the filtered maps they induce on $\CKh(K)$. This is identical to the bound stated in Theorem~\ref{slicegenus} when $0 \le \alpha \le 1$, but is significantly weaker when $1 \le \alpha \le 2$. To prove the stronger bound, more complex arguments are needed that involve modifying the knot to simplify it's Khovanov homology and using Proposition~\ref{crossingchange} to control the change in $\Phi$. 

Let $\Sigma \subset S^3$ be a ribbon surface of genus $g$, with a handle decomposition consisting of $k$ discs and $k + 2g - 1$ disjoint bands. Via slides among the bands, we may assume that the ribbon graph formed by these discs and bands has the same combinatorial type as the one visible in Figure~\ref{finalknot}. For $n \ge 0$, let $\Sigma_n$ be the surface obtained from $\Sigma$ by inserting $n$ full positive twists into each band of $\Sigma$, and let $K_n = \partial \Sigma_n$. 

\begin{prop}\label{twistcontrol}
For fixed $\alpha$, the quantity $\Phi_{K_n}(\alpha)$ is monotonically decreasing in $n$, and if $\alpha \neq 1$ then $\lim_{n \to \infty} \Phi_{K_n}(\alpha) = -2g\Lambda(\alpha)$. 
\end{prop}

Theorem~\ref{slicegenus} follows quickly from this result, and the core of its proof involves understanding the delta graded Khovanov homology of $K_n$ for large $n$. Here and throughout this paper, the grading convention $\delta = q - 2h$ for the delta grading is used. 

\begin{lemma}
In any fixed homological grading, the Khovanov homology of $K_n$ is supported only in delta grading $-2g$ for sufficiently large $n$. Furthermore, there is a finite range of delta gradings that contains the entire support of the Khovanov homology of any $K_n$. 
\end{lemma}
\begin{proof}
First, to see that $\operatorname{Kh}(K_n)$ is homologically thin in small homological gradings for large $n$, consider the knot $K_\infty \subset M = \#^{k + 2g - 1} S^2 \times S^2$ formed by performing a $0$-surgery on a meridian of each of the bands of rhe ribbon surface $\Sigma$, and the Khovanov homology $\operatorname{Kh}(K_\infty)$ as defined by Willis in \cite{willis2019khovanov}. By Theorem~2.2 in \cite{willis2019khovanov}, for sufficiently large $n$ the Khovanov homologies $\operatorname{Kh}(K_n)$ and $\operatorname{Kh}(K_\infty)$ agree in any given homological degree. Therefore, it suffices to show that $\operatorname{Kh}(K_\infty)$ is supported in delta grading $-2g$. However, by isotopies involving sliding the knot $K_\infty$ over surgery curves for $M$, the bands of $\Sigma$ can be passed through eachother and the boundaries of the discs, so $K_\infty$ is isotopic to the knot shown in Figure~\ref{finalknot}, which in particular depends only on $g$ and not on the embedding of $\Sigma$. Using again the fact that $\operatorname{Kh}(K_\infty)$ agrees in any fixed homological grading with the result of replacing each $0$-surgery with a large but finite number of twists but in this new diagram, $\operatorname{Kh}(K_\infty)$ agrees with the Khovanov homology of the connected sum of pretzel knots $\#^g P(2n+1,2n+1,1)$ for large enough $n$ in any fixed homological grading. This knot is alternating and has signature $-2g$, so $\operatorname{Kh}(\#^g P(2n+1,2n+1,1))$ is supported in delta grading $-2g$ for any $n$ and the same is true of $\operatorname{Kh}(K_\infty)$.

To see that the homologies $\operatorname{Kh}(K_n)$ are supported in a uniformly bounded range of delta gradings, bounds based on the spanning tree complex for Khovanov homology suffice. $K_0$ has a diagram $D_0$ containing $k + 2g - 1$ distinguished negative crossings, corresponding to the bands in the handle decomposition of $\Sigma$, such that resolving these crossings produces an unlink with $k$ components and diagrams $D_n$ of each $K_n$ can be obtained from $D_0$ by replacing each of these crossings with $2n + 1$ negative crossings arranged in a row as in Figure~\ref{twistregion}. Let $n_-$, $n_+$, $O_A$, and $O_B$ be the number of negative crossings, positive crossings, circles in the all-A resolution, and circles in the all-B resolution of $D_0$. Then $D_n$ has $n_-+2n$ negative crossings, $n_+$ positive crossings, $O_A$ circles in the all-A resolution, and $O_B+2n$ in the all-B resolution. Therefore, the reduced Khovanov homologies of all of the $K_n$ are supported in delta gradings between $O_A - n_+ - 1$ and $-O_B + n_- + 1$, via the bounds proven in \cite{dasbach2011turaev} using a spanning tree complex. 

\end{proof}

\begin{figure}[h]
  \centering
  \def\svgwidth{\columnwidth}
  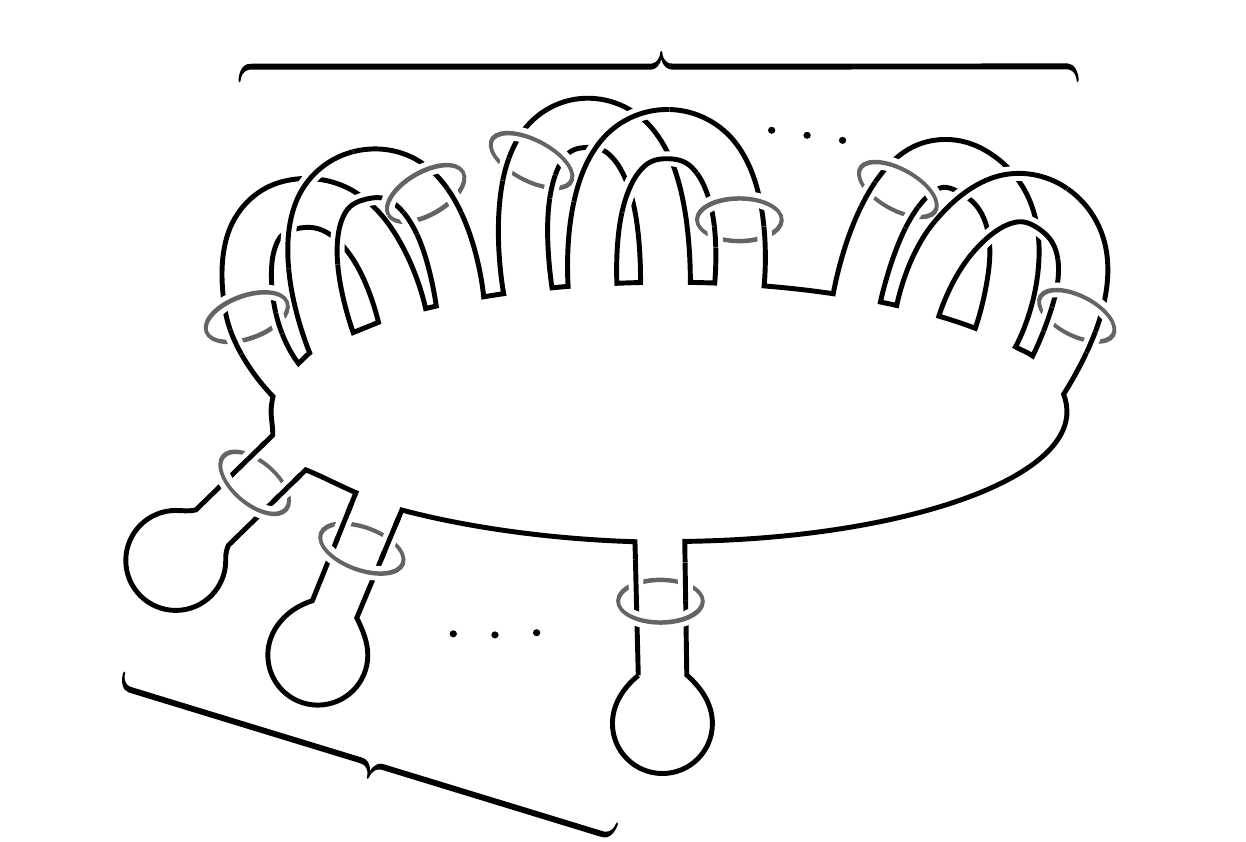
  \caption{The general form to which $K_\infty$ can be isotoped. The black curve is the knot, and the grey curves are the ($0$-framed) surgery circles giving $S^2 \times S^1$.}\label{finalknot}
\end{figure}

\begin{figure}[h]
  \centering
  \def\svgwidth{\columnwidth}
  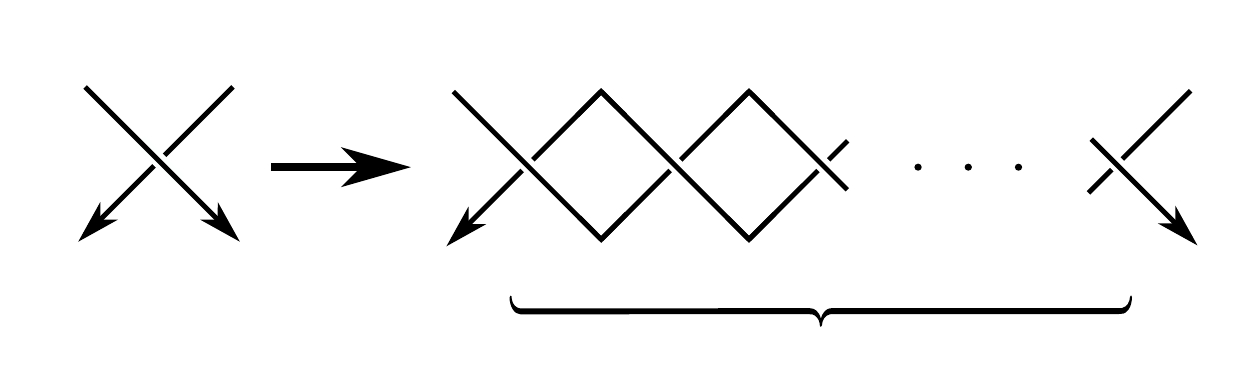
  \caption{The local model for the twisting producing $K_n$ from $K = K_0$.}\label{twistregion}
\end{figure}

\begin{cor}\label{limitval}
For any $\epsilon > 0$, there is an $N$ for which, whenever $n \ge N$ and $|\alpha - 1| > \epsilon$, the invariant $\Phi_{K_n}(\alpha)$ is equal to $-2g\Lambda(\alpha)$.
\end{cor}
\begin{proof}
By the previous Lemma, there is a constant $M$ for which the reduced Khovanov homologies of the knots $K_n$ are supported in delta gradings between $-M$ and $M$ and, by taking $n$ large enough, can be taken to be supported in delta grading $-2g$ in any finite range of homological gradings.

By Theorem~\ref{homologyconstraint}, then, the graphs of $\Phi_{K_n}$ for sufficiently large $n$ are contained in the a union of lines that all either pass through $(1,-2g)$ or meet the line $\alpha = 1$ between $-M$ and $M$ and have slope as large as desired. A sufficiently steep line meeting $\alpha = 1$ between $-M$ and $M$ cannot pass through $(0,0)$ or $(2,0)$, so since $\Phi_{K_n}(0) = \Phi_{K_n}(2) = 0$, near the endpoints the graph of $\Phi_{K_n}$ must follow either the line $(\alpha, -2g\alpha)$ or $(\alpha, -2g(2-\alpha))$. Since $\Phi_{K_n}$ is continuous, the graph must follow these lines at least until they meet another one of the supporting lines. If this happens for at a point with $\alpha \neq 1$, it must be because one of the initial lines meets one of the very steep lines. By taking these additional lines steep enough, this can be restricted to happen only in a small neighborhood of $\alpha = 1$. 
\end{proof}

\begin{proof}[Proof of Proposition~\ref{twistcontrol}]
$K_{n+1}$ is obtained from $K_n$ by changing a positive crossing to a negative one (since the two strands in a band are oppositely oriented) so Proposition~\ref{crossingchange} shows that $\Phi_{K_n}(\alpha)$ is monotonically decreasing, and Corollary~\ref{limitval} shows that $\Phi_{K_n}(\alpha)$ is eventually equal to $-2g\Lambda(\alpha)$ if $\alpha \neq 1$.
\end{proof}

\begin{proof}[Proof of Theorem~\ref{slicegenus}]
If $K$ has slice genus $g$, then it is concordant to a knot bounding a ribbon surface of genus $g$, so since $\Phi$ is a concordance invariant we may assume that $K$ bounds a ribbon surface of genus $g$. With $K_n$ as in Proposition~\ref{twistcontrol}, then, for any $\alpha \neq 1$ $$\Phi_K(\alpha) = \Phi_{K_0}(\alpha) \ge \lim_{n \to \infty} \Phi_{K_n}(\alpha) = -2g\Lambda(\alpha),$$
and the same must be true for $\alpha = 1$ since $\Phi_K$ is a continuous function of $\alpha$. The theorem now follows from applying this inequality to both $K$ and $-K$. 
\end{proof}

\section{Pretzel calculations and a conjecture on torus knots}

\begin{prop}\label{pretzelhomology}
If $K$ is the pretzel knot $P(-2n,a,b)$ with $a$ and $b$ odd and $0 < 2n < a,b$, then the only singularities of $\Phi_{K}(\alpha)$ are at $\alpha = 1$ and $\alpha = \frac{8n}{8n+2}$. The change in the slope of $\Phi_K(\alpha)$ at $\alpha = \frac{8n}{8n+2}$ is exactly
\begin{equation*}
\Delta\Phi'_K\left(\frac{8n}{8n+2}\right) = -8n-2
\end{equation*}
\end{prop}
\begin{proof}
By Manion's computations of the reduced Khovanov homology of pretzel knots \cite{manion2018khovanov}, the reduced homology of $K$ is supported in delta gradings $a+b$ and $a+b - 2$ (note that here we use the convention $\delta = q-2h$, which differs from Manion's by a factor of $2$). In the higher delta grading the homology is supported between homological gradings $0$ and $4n+1$, and in the lower delta grading the homology is supported in homological gradings at least $4n$. Furthermore, the homology group in delta grading $a+b$ and homological grading $4n+1$ has rank $1$, as does the homology group in delta grading $a+b-2$ and homological grading $4n$. Using the model provided by Lemma~\ref{fullsimplify}, this is enough to fully compute $\CKh(K)$. For each of the Lee and $E(-1)$ differentials, there is only one possibility consistent with the facts that the overall homology is one dimensional and the surviving generator is in homological or internal grading $0$. There is no room in the reduced Khovanov homology of $K$ for any differentials going deeper into both filtrations, and only one way up to isomorphism of the resulting complex to match the generators of the Lee and $E(-1)$ complexes to make $d^2 = 0$, so the differential on $\CKh(K)$ is just the sum of the Lee and $E(-1)$ differentials. Overall, the complex $\CKh(K)$ is shown in figure~\ref{pretzelcomp}; it is the direct sum of one main staircase complex with copies of shifts of the boxed contractible complex. The boxed complex is homotopy equivalent to $0$, so in fact $\CKh(K)$ is homotopy equivalent to the main complex $C$ in Figure~\ref{pretzelcomp}. This means that $\Phi_K(\alpha) = M(C,\alpha)$, which can be computed to be
\begin{equation*}
\Phi_K(\alpha) = \begin{cases}
(a+b)\alpha & 0 \le \alpha \le \frac{8n}{8n+2} \\
8n + (a+b-8n-2)\alpha & \frac{8n}{8n+2} \le \alpha \le 1 \\
(a+b-2)(2-\alpha) & 1 \le \alpha \le 2
\end{cases}
\end{equation*}
\end{proof}

\begin{lemma}\label{deltaintegral}
For any knot $K$, the quantity $\Delta \Phi'_K(\alpha)(1-\alpha)$ is always an even integer. 
\end{lemma}
\begin{proof}
By Theorem~\ref{homologyconstraint}, this is a difference between two delta gradings in which the reduced Khovanov homology of $K$ is supported. The parities of the delta and quantum gradings are equal, and the reduced Khovanov homology of a knot is supported in even quantum gradings. 
\end{proof}

\begin{proof}[Proof of Theorem~\ref{pretzelindep}]
For $n \ge 0$ an integer, let $$i_n(K) = \frac{\Delta\Phi'_K\left(\frac{8n}{8n+2}\right)}{8n+2}.$$ By Lemma~\ref{deltaintegral}, this is always an integer, and since $i_n$ depends linearly on $\Phi$ it is a concordance homomorphism. Finally, by Proposition~\ref{pretzelhomology}, $i_n(P(-2k,a,b)) = -1$ if $n = k$ and is zero otherwise. 
\end{proof}

\begin{figure}[h]
  \centering
  \def\svgwidth{\columnwidth}
  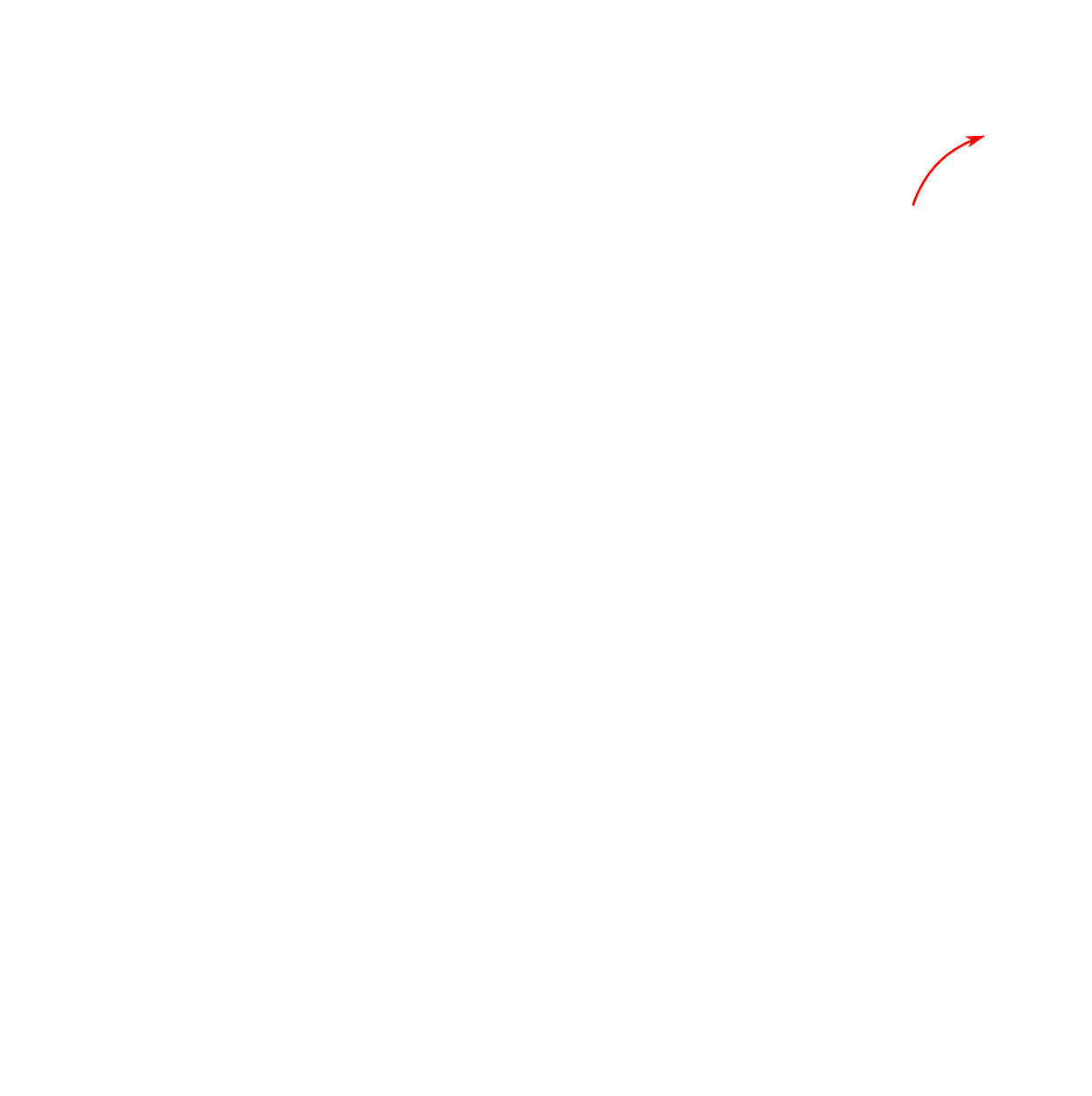
  \caption{The complex $\CKh(P(-2n,a,b))$ is the direct sum of the larger complex in this figure with several shifts of the smaller, boxed complex. All arrows shown are terms in the diferential of $\CKh(P(-2n,a,b))$, with the ones in blue coming from the $E(-1)$ spectral sequence and the ones in red coming from the Lee spectral sequence. Some gradings in the larger complex are given by the exponents of $q$ and $h$ (the standard quantum and homological gradings), and the rest may be inferred from the fact that all $E(-1)$ differentials present have degree $q^6h^2$ and, apart from the one between gradings $q^{a+b+8n-2}h^{4n}$ and $q^{a+b+8n+2}h^{4n+1}$, the Lee differentials have degree $q^2h$. }\label{pretzelcomp}
\end{figure}

Similar techniques may be used to compute the invariant $\Phi$ for individual small knots, although not yet for any other infinte family of non-thin knots. However, looking at the values of $\Phi$ on small torus knots, a clear pattern emerges leading to a conjecture for the value of $\Phi$ on all torus knots:
\begin{conj}\label{torusvalues}
For the torus knot $T_{n,n+1}$, the jump in the slope $\Delta \Phi'_{T_{n,n+1}}(\alpha)$ is equal to $0$ except when $\alpha = \frac{2+2k}{2+3k}$ for $0 \le k < n$ and $k$ of the same parity as $n$, when it is $-4-6k$. Together with the fact that $s(T_{n,n+1}) = n(n-1)$, this determines $\Phi_{T_{n,n-1}}$. For other torus knots, $\Phi$ is determined by the recurrence
\begin{equation*}
\Phi_{T_{p,q+p}} = \Phi_{T_{p,q}} + \Phi_{T_{p,p+1}}
\end{equation*}
\end{conj}
This conjecture can be confirmed for small $p,q$ using computer calculations of $\operatorname{Kh}(T_{p,q})$ from \cite{bar2005knot}, and similarly for $p = 2,3$ and any $q$ since the Khovanov homologies of these knots are known (see \cite{turner2008spectral} for the $(3,q)$ case). It is also consistent with the calculations of $t(T_{p,q}) = -\Phi'_{T_{p,q}}(2)$ from \cite{ballinger2020concordance}. The conjectured recurrence relation $\Phi_{T_{p,q+p}} = \Phi_{T_{p,q}} + \Phi_{T_{p,p+1}}$ is known to be satisfied by the Upsilon invariant \cite{feller2017cobordisms}, but the values of the two invariants on $T_{n,n+1}$ are different. Explaining the recurrence satisfied by both $\Upsilon$ and (conjecturally) $\Phi$ would be an interesting exploration into the relationships between knot Floer homology and Khovanov homology.

\bibliographystyle{amsplain}
\bibliography{filteredcomp}

\end{document}

%% file: differentialregions.pdf_tex
\begingroup%
  \makeatletter%
  \providecommand\color[2][]{%
    \errmessage{(Inkscape) Color is used for the text in Inkscape, but the package 'color.sty' is not loaded}%
    \renewcommand\color[2][]{}%
  }%
  \providecommand\transparent[1]{%
    \errmessage{(Inkscape) Transparency is used (non-zero) for the text in Inkscape, but the package 'transparent.sty' is not loaded}%
    \renewcommand\transparent[1]{}%
  }%
  \providecommand\rotatebox[2]{#2}%
  \newcommand*\fsize{\dimexpr\f@size pt\relax}%
  \newcommand*\lineheight[1]{\fontsize{\fsize}{#1\fsize}\selectfont}%
  \ifx\svgwidth\undefined%
    \setlength{\unitlength}{358.56000137bp}%
    \ifx\svgscale\undefined%
      \relax%
    \else%
      \setlength{\unitlength}{\unitlength * \real{\svgscale}}%
    \fi%
  \else%
    \setlength{\unitlength}{\svgwidth}%
  \fi%
  \global\let\svgwidth\undefined%
  \global\let\svgscale\undefined%
  \makeatother%
  \begin{picture}(1,0.73317093)%
    \lineheight{1}%
    \setlength\tabcolsep{0pt}%
    \put(0,0){\includegraphics[width=\unitlength,page=1]{differentialregions.pdf}}%
    \put(0.26086268,0.02949619){\color[rgb]{0,0,0}\makebox(0,0)[lt]{\lineheight{1.25}\smash{\begin{tabular}[t]{l}$1$\end{tabular}}}}%
    \put(0.30946017,0.0887009){\color[rgb]{0,0,0}\makebox(0,0)[lt]{\lineheight{1.25}\smash{\begin{tabular}[t]{l}$q^2h$\end{tabular}}}}%
    \put(0.42575662,0.21023718){\color[rgb]{0,0,0}\makebox(0,0)[lt]{\lineheight{1.25}\smash{\begin{tabular}[t]{l}$q^6h^3$\end{tabular}}}}%
    \put(0.30878253,0.15144201){\color[rgb]{0,0,0}\makebox(0,0)[lt]{\lineheight{1.25}\smash{\begin{tabular}[t]{l}$q^4h$\end{tabular}}}}%
    \put(0.30983898,0.2137737){\color[rgb]{0,0,0}\makebox(0,0)[lt]{\lineheight{1.25}\smash{\begin{tabular}[t]{l}$q^6h$\end{tabular}}}}%
    \put(0.5547469,0.39865557){\color[rgb]{0,0,0}\makebox(0,0)[lt]{\lineheight{1.25}\smash{\begin{tabular}[t]{l}$q^{12}h^5$\end{tabular}}}}%
    \put(0.68459444,0.58569167){\color[rgb]{0,0,0}\makebox(0,0)[lt]{\lineheight{1.25}\smash{\begin{tabular}[t]{l}$q^{18}h^7$\end{tabular}}}}%
  \end{picture}%
\endgroup%

%% file: generalS2S1.pdf_tex
\begingroup%
  \makeatletter%
  \providecommand\color[2][]{%
    \errmessage{(Inkscape) Color is used for the text in Inkscape, but the package 'color.sty' is not loaded}%
    \renewcommand\color[2][]{}%
  }%
  \providecommand\transparent[1]{%
    \errmessage{(Inkscape) Transparency is used (non-zero) for the text in Inkscape, but the package 'transparent.sty' is not loaded}%
    \renewcommand\transparent[1]{}%
  }%
  \providecommand\rotatebox[2]{#2}%
  \newcommand*\fsize{\dimexpr\f@size pt\relax}%
  \newcommand*\lineheight[1]{\fontsize{\fsize}{#1\fsize}\selectfont}%
  \ifx\svgwidth\undefined%
    \setlength{\unitlength}{358.72558594bp}%
    \ifx\svgscale\undefined%
      \relax%
    \else%
      \setlength{\unitlength}{\unitlength * \real{\svgscale}}%
    \fi%
  \else%
    \setlength{\unitlength}{\svgwidth}%
  \fi%
  \global\let\svgwidth\undefined%
  \global\let\svgscale\undefined%
  \makeatother%
  \begin{picture}(1,0.67596194)%
    \lineheight{1}%
    \setlength\tabcolsep{0pt}%
    \put(0,0){\includegraphics[width=\unitlength,page=1]{generalS2S1.pdf}}%
    \put(0.53779879,0.64647045){\color[rgb]{0,0,0}\makebox(0,0)[lt]{\lineheight{1.25}\smash{\begin{tabular}[t]{l}$g$\end{tabular}}}}%
    \put(0.24365188,0.02377548){\color[rgb]{0,0,0}\makebox(0,0)[lt]{\lineheight{1.25}\smash{\begin{tabular}[t]{l}$k-1$\end{tabular}}}}%
  \end{picture}%
\endgroup%

%% file: twistband.pdf_tex
\begingroup%
  \makeatletter%
  \providecommand\color[2][]{%
    \errmessage{(Inkscape) Color is used for the text in Inkscape, but the package 'color.sty' is not loaded}%
    \renewcommand\color[2][]{}%
  }%
  \providecommand\transparent[1]{%
    \errmessage{(Inkscape) Transparency is used (non-zero) for the text in Inkscape, but the package 'transparent.sty' is not loaded}%
    \renewcommand\transparent[1]{}%
  }%
  \providecommand\rotatebox[2]{#2}%
  \newcommand*\fsize{\dimexpr\f@size pt\relax}%
  \newcommand*\lineheight[1]{\fontsize{\fsize}{#1\fsize}\selectfont}%
  \ifx\svgwidth\undefined%
    \setlength{\unitlength}{358.72558594bp}%
    \ifx\svgscale\undefined%
      \relax%
    \else%
      \setlength{\unitlength}{\unitlength * \real{\svgscale}}%
    \fi%
  \else%
    \setlength{\unitlength}{\svgwidth}%
  \fi%
  \global\let\svgwidth\undefined%
  \global\let\svgscale\undefined%
  \makeatother%
  \begin{picture}(1,0.30125179)%
    \lineheight{1}%
    \setlength\tabcolsep{0pt}%
    \put(0,0){\includegraphics[width=\unitlength,page=1]{twistband.pdf}}%
    \put(0.6298338,0.00668231){\color[rgb]{0,0,0}\makebox(0,0)[lt]{\lineheight{1.25}\smash{\begin{tabular}[t]{l}$2n+1$\end{tabular}}}}%
  \end{picture}%
\endgroup%

%% file: pretzelcomplex.pdf_tex
\begingroup%
  \makeatletter%
  \providecommand\color[2][]{%
    \errmessage{(Inkscape) Color is used for the text in Inkscape, but the package 'color.sty' is not loaded}%
    \renewcommand\color[2][]{}%
  }%
  \providecommand\transparent[1]{%
    \errmessage{(Inkscape) Transparency is used (non-zero) for the text in Inkscape, but the package 'transparent.sty' is not loaded}%
    \renewcommand\transparent[1]{}%
  }%
  \providecommand\rotatebox[2]{#2}%
  \newcommand*\fsize{\dimexpr\f@size pt\relax}%
  \newcommand*\lineheight[1]{\fontsize{\fsize}{#1\fsize}\selectfont}%
  \ifx\svgwidth\undefined%
    \setlength{\unitlength}{358.72558594bp}%
    \ifx\svgscale\undefined%
      \relax%
    \else%
      \setlength{\unitlength}{\unitlength * \real{\svgscale}}%
    \fi%
  \else%
    \setlength{\unitlength}{\svgwidth}%
  \fi%
  \global\let\svgwidth\undefined%
  \global\let\svgscale\undefined%
  \makeatother%
  \begin{picture}(1,1.00355262)%
    \lineheight{1}%
    \setlength\tabcolsep{0pt}%
    \put(0.90349598,0.8572012){\color[rgb]{0,0,0}\makebox(0,0)[lt]{\lineheight{1.25}\smash{\begin{tabular}[t]{l}$q^{3a+3b-4}h^{a+b-1}$\end{tabular}}}}%
    \put(0,0){\includegraphics[width=\unitlength,page=1]{pretzelcomplex.pdf}}%
    \put(0.02090735,0.0418147){\color[rgb]{0,0,0}\makebox(0,0)[lt]{\lineheight{1.25}\smash{\begin{tabular}[t]{l}$q^{a+b}$\end{tabular}}}}%
    \put(0,0){\includegraphics[width=\unitlength,page=2]{pretzelcomplex.pdf}}%
    \put(0.36396822,0.60663136){\color[rgb]{0,0,0}\makebox(0,0)[lt]{\lineheight{1.25}\smash{\begin{tabular}[t]{l}$q^{a+b+8n+2}h^{4n+1}$\end{tabular}}}}%
    \put(0.49206934,0.44950795){\color[rgb]{0,0,0}\makebox(0,0)[lt]{\lineheight{1.25}\smash{\begin{tabular}[t]{l}$q^{a+b+8n-2}h^{4n}$\end{tabular}}}}%
    \put(0,0){\includegraphics[width=\unitlength,page=3]{pretzelcomplex.pdf}}%
    \put(0.64812775,0.82584018){\color[rgb]{0,0,0}\makebox(0,0)[lt]{\lineheight{1.25}\smash{\begin{tabular}[t]{l}$q^{3a+3b-6}h^{a+b-2}$\end{tabular}}}}%
    \put(0,0){\includegraphics[width=\unitlength,page=4]{pretzelcomplex.pdf}}%
  \end{picture}%
\endgroup%